\documentclass[12pt]{amsart}
\usepackage{amsfonts}
\usepackage{amssymb,amsthm,amscd,amsfonts,graphicx,tabularx,array,CJK,fancyhdr}
\usepackage{amsmath}
\usepackage{stmaryrd}
\usepackage{txfonts}
\usepackage{bbm}
\usepackage{amsfonts}
\usepackage{amssymb}
\usepackage{mathrsfs}
\usepackage{geometry}
\usepackage{times}
\usepackage{cases}
\usepackage{graphicx}
\usepackage{ifxetex}
\usepackage[arc,knot,all]{xy}
\usepackage{comment}
\usepackage{amsgen, amsthm, amscd, xspace}
\usepackage{hyperref}

\newcommand{\R}{{\mathbb{R}}}
\newcommand{\C}{{\mathbb{C}}}
\newcommand{\Q}{{\mathbb{Q}}}
\newcommand{\N}{{\mathbb{N}}}

\newcommand{\mF}{{\mathbb{F}}}

\newcommand{\diag}{{\mathrm d}{\mathrm i}{\mathrm a}{\mathrm g\,}}
\newcommand{\Det}{{\mathrm d}{\mathrm e}{\mathrm t\,}}

\newcommand{\rank}{{\mathrm r}{\mathrm a}{\mathrm n}{\mathrm k\,}}

\newcommand{\GL}{{\operatorname{GL}}}
\newcommand{\SL}{{\operatorname{SL}}}

 \makeatother

\newtheorem{thm}{Theorem}[section]
\newtheorem{defn}[thm]{Definition}

\newtheorem{prop}[thm]{Proposition}

\newtheorem{lem}[thm]{Lemma}
\newtheorem{rem}[thm]{Remark}

\begin{document}

\title{Multiplicative maps on matrix algebras}
\author{Xiaomei Yang}
\address[Yang]{School of Mathematics, Sichuan Normal University, Chengdu, Sichuan 610068, P.~R.~China}
\email{xiaomeiyang@mail.nankai.edu.cn}
\author{Fuhai Zhu}
\address[Zhu]{Department of Mathematics, Nanjing University, Nanjing, P.~R.~China}
\email{zhufuhai@nju.edu.cn}

\maketitle
Abstract:
In this paper, we use elementary method to give a classification of the multiplicative
maps on matrix algebra $M_{n}(\mF)$ over a field $\mF$ of characteristic $0$.
All the multiplicative maps are classified into three classes: the trivial ones, the degenerate ones and the non-degenerate ones.



\date{} \maketitle
\section{Introduction and statement of the main result}

During the past one hundred years, a great deal of mathematical effort has been devoted to the study of preserver problems in matrix theory. In \cite{DS2002}, \cite{JS1986} and \cite{Oml1986}, the authors deal with linear preserver problems, i.e., linear maps on matrix algebras and operator algebras preserving determinant, spectrum, commutativity, respectively. Chi-Kwong Li and Nam-Kiu Tsing presented a brief description of the subject in \cite{LT1992}.

Naturally, some similar preserver problems are raised and investigated for multiplicative maps (i.e., maps $\Phi$ satisfying $\Phi(AB)=\Phi(A)\Phi(B)$) instead of linear maps.
In \cite{Hoch1994}, Hochwald described the form of those multiplicative maps of a matrix algebra which preserve the spectrum.
As a natural generalization, he also raised the question of spectrum-preserving multiplicative maps for infinite-dimensional vector spaces with added condition of surjectivity, see \cite{OS1991}. In \cite{AH2008}, An and Hou considered rank preserving multiplicative maps on matrix algebra.
In \cite{CFL2002}, the authors established some general techniques that are used to prove results for multiplicative maps on semigroups of square matrices and obtained some basic results to study multiplicative preserver problems.
In recent years, there have been extensive study and application of multiplicative preserver problems on operator algebras over infinite-dimensional spaces. In \cite{Hou1998},
Hou gave a structure theorem of abstract multiplicative maps $\Phi$ from $\mathfrak{B}(X)$ into $\mathfrak{B}(Y)$ which sends some
rank-$1$ operator to an operator of rank not greater than $1$.
Afterwards, An and Hou obtained some characterizations of rank-preserving
multiplicative maps on $\mathfrak{B}(X)$ in \cite{AH2002}. In \cite{M1999},
Moln\'{a}r described the form of those continuous multiplicative
maps on $\mathfrak{B}(H)$ ($H$ is a separable complex Hilbert space of dimension not less than 3) which preserve the rank or the corank.

It is very strange that people always try to deal with multiplicative maps with some additional conditions, say, preserving spectrum, ranks, etc., rather than to classify all multiplicative maps first. In \cite{JL1968}, the authors investigated non-degenerate multiplicative maps (see section~\ref{M(n)} for the definition) of $n\times n$ matrices over a principal ideal domain $R$. In \cite{Sem2008}, the author
considered non-degenerate multiplicative maps of $n\times n$ matrices over a division ring. The purpose of this paper is to deal with the multiplicative maps for matrix algebras over any field of characteristic zero with an elementary method, which may be applicable to positive characteristic case.
All the multiplicative maps are separated into three classes: we devote Section \ref{sec-2} to study the trivial ones, which maps $\SL(n,\mF)$ to the identity. We prove that such multiplicative map is reduced to a multiplicative map $\overline{\Phi}:\mF\rightarrow M_k(\mF)$ with $\overline{\Phi}(1)=I_k$. In particular, we show that any of multiplicative map $\Phi:M_{n}(\mF)\rightarrow M_{k}(\mF)(k<n)$ is actually a trivial multiplicative map; in Section~\ref{GL}, we classify the degenerate ones, which maps every singular matrix to the zero matrix. We present such $\Phi(A)$ is the form of $\lambda(\det(\Phi_1(A)))\Phi_2(A)$, where $\lambda:\mF^{\ast}\rightarrow \mF^{\ast}$ is a multiplicative map and $\Phi_1,\Phi_2$ are compositions of the following multiplicative maps: $(a). \Phi_R(A)=R^{-1}AR$ for some $n\times n$ invertible matrix $R$;
$(b). \Phi^*(A)=(A^{\ast})^{t}$, where $A^\ast$ is the cofactor matrix of $A$ and $A^t$ is the transpose of $A$; $(c). \Phi_\varphi(A)=(\varphi(a_{ij}))$
for any $A=(a_{ij})\in \GL(n)$, where $\varphi:\mF\rightarrow \mF$ is a ring homomorphism ; the non-degenerate multiplicative maps are considered in Section \ref{M(n)}, in which we
conclude that $\Phi(A)=R^{-1}(\varphi(a_{ij}))R$ or $\Phi(A)=R^{-1}(\varphi(a_{ij})^{\ast})^{t}R,$
for some ring homomorphism $\varphi:\mF\rightarrow \mF$
and some invertible matrix $R$. With such classification in hand, it it easy to find all multiplicative maps which preserve spectrums, ranks, etc.

\section{Trivial Multiplicative Maps}\label{sec-2}

Denote by $M_n(\mF)$ the matrix algebra over a field $\mF$ of characteristic zero.
A map $\Phi: M_{n}(\mF)\rightarrow M_{k}(\mF)$ is called a \emph{multiplicative map} if $\Phi(AB)=\Phi(A)\Phi(B)$ for
any $A,B\in M_{n}(\mF)$.
In this section, we study a special case of multiplicative maps. First, we need the following easy property of multiplicative maps.

\begin{lem}
Let $\Phi:M_{n}(\mF)\rightarrow M_{k}(\mF)$ be a multiplicative map.
Then for any matrix $A\in M_{n}(\mF)$, up to a conjugation, we have
\[\Phi(A)=\left(
                \begin{array}{ccc}
                  \widetilde{A} & 0 & 0 \\
                  0 & 0 & 0 \\
                  0 & 0 & I_{s} \\
                \end{array}
              \right),\]
where $s=\rank(\Phi(0)), \ l=\rank(\Phi(I_{n}))-\rank(\Phi(0))$
and $\widetilde{A}\in M_{l}(\mF)$ .
\end{lem}
\begin{proof}
Since $\Phi(0),\Phi(I_{n})$ are idempotent matrices. We may assume that $\Phi(0),\Phi(I_{n})$ are diagonal matrices with the
diagonal elements being $0$ or $1$. Using
$\Phi(0)\Phi(I_{n})=\Phi(I_{n})\Phi(0)=\Phi(0)$,
up to a conjugation, we have
\[\Phi(0)=\left(
            \begin{array}{cc}
              0 & 0 \\
              0 & I_{s} \\
            \end{array}
          \right)
, \ \ \Phi(I_{n})=\left(
                    \begin{array}{ccc}
                      I_{l} & 0 & 0 \\
                      0 & 0 & 0 \\
                      0 & 0 & I_{s} \\
                    \end{array}
                  \right),
\]
where $s=\rank(\Phi(0)),l=\rank(\Phi(I_{n}))-\rank(\Phi(0))$.

For any $A\in M_{n}(\mF)$, from
\[\Phi(0)\Phi(A)=\Phi(A)\Phi(0)=\Phi(0),\]
\[\Phi(I_{n})\Phi(A)=\Phi(A)\Phi(I_{n})=\Phi(A),\]
it follows that $\Phi(A)=\left(
                \begin{array}{ccc}
                  \widetilde{A} & 0 & 0 \\
                  0 & 0 & 0 \\
                  0 & 0 & I_{s} \\
                \end{array}
              \right)$,
where $\widetilde{A}$ is a $l\times l$ matrix.
\end{proof}

From the proof of the above Lemma, one can define a multiplicative map $\overline{\Phi}:M_{n}(\mF)\rightarrow M_{l}(\mF)$ by $\overline{\Phi}(A)=\widetilde{A}$. It is easy to see that $\Phi$ is determined by $\overline{\Phi}$ and $\rank(\Phi(0))$ uniquely. In the remainder of this paper, we only consider those maps such that $\Phi(0)=0$ and $\Phi(I_n)=I_k$. Then it is trivial that $\Phi(A^{-1})=\Phi(A)^{-1}$ and $\Phi^{-1}(I_k)$ is a normal subgroup of $\GL(n,\mF)$. A classical result says that if $N$ is a normal subgroup of $\GL(n,\mF)$, then either $N$ is contained in $Z(\GL(n,\mF))$, that is the center of $\GL(n,\mF)$, or $N$ contains $\SL(n,\mF)$. It is natural to introduce the following definition.

\begin{defn}
A multiplicative map $\Phi:M_{n}(\mF)\rightarrow M_{k}(\mF)$ is called trivial if $\SL(n,\mF)\subseteq \Phi^{-1}(I_{k})$.
\end{defn}

Trivial multiplicative maps have the following property.

\begin{lem}
Let $\Phi: M_{n}(\mF)\rightarrow M_{k}(\mF)$ be a trivial multiplicative map. Then $\Phi(A)=0$ if $\rank(A)<n$.
\end{lem}
\begin{proof}
Since $\rank(A)=r<n$, there exist $P,Q,S,T\in \SL(n,\mF)$ such that $A=PIQ=SJT$, where $I=\left(\begin{array}{cc}I_r&0\\0&0\end{array}\right)$ and $J$ is a Jordan canonical form with diagonal entries being $0$ and $\rank(J)=r$. Thus $\Phi(A)=\Phi(I)=\Phi(J)$. Since $I=I^m$ and $J^m=0$ for $m>r$, it is easy to see that $\Phi(A)=\Phi(I)=\Phi(I^m)=\Phi(J^m)=0$.
\end{proof}

Thanks to the above lemma, any trivial multiplicative map $\Phi$ is reduced to a multiplicative map $\overline{\Phi}:\mF\rightarrow M_k(\mF)$ with $\overline{\Phi}(1)=I_k$. The following proposition shows that many multiplicative maps are trivial and it is crucial for rest of the paper.

\begin{prop}\label{Ker1}
Let $\Phi:M_{n}(\mF)\rightarrow M_{k}(\mF)(k<n)$
be a multiplicative map.
Then $\Phi$ is a trivial multiplicative map.
\end{prop}
\begin{proof}
Let $N$ (resp. $\overline{N}$) denote the set of $n\times n$
(resp. $k\times k$) upper triangular matrices
with all the diagonal entries being $1$. Consider the lower central sequence: $N^{0}=N$, $N^{1}=[N,N]=\{xyx^{-1}y^{-1}|x,y\in N\},\ldots,N^{i}=[N,N^{i-1}]$.
It is easy to calculate that $N^{n}=\{I_{n}\}$ and
$(\overline{N})^{k}=\{I_{k}\}$.
Furthermore, since $(\Phi(N))^{n}=\Phi(N^{n})=\Phi(I_{n})=I_{k}$,
it follows that, up to a conjugation, $\Phi(N)\subseteq \overline{N}$.
Then we have $\Phi(N)^{k}\subseteq (\overline{N})^{k}=\{I_{k}\}$,
which implies that
$N^{k}\subseteq \Phi^{-1}(I_{k})$.
Therefore $\Phi^{-1}(I_k)\supseteq \SL(n,\mF)$ since it is not contained in the center of $\GL(n,\mF)$.
\end{proof}

Thus, any of multiplicative map $\Phi:M_{n}(\mF)\rightarrow M_{k}(\mF)(k<n)$
is the form of $\Phi(A)=\overline{\Phi}(\Det(A))$ for some multiplicative map $\overline{\Phi}:\mF\rightarrow M_{k}(\mF)$.

\section{Degenerate Multiplicative Maps}\label{GL}

In this section, we just need to consider non-trivial multiplicative maps and
hence we have $\Phi(0)=0$ and $\Phi(I_{n})=I_{n}$.
Following Jodeit and Lam we call a nontrivial multiplicative map $\Phi:M_{n}(\mF)\rightarrow M_{n}(\mF)$ \emph{degenerate}, if $\Phi(A)=0$ for any degenerate $A$, and is called \emph{non-degenerate} otherwise.
It is trivial that degenerate multiplicative maps are in 1-1 correspondence with multiplicative maps
$\Phi:\GL(n,\mF)\rightarrow \GL(n,\mF)$.

\begin{thm}\label{GL(n)}
Let $\Phi:\GL(n,\mF)\rightarrow \GL(n,\mF)$ be a non-trivial
multiplicative map. Then $\Phi(A)=\lambda(\det(\Phi_1(A)))\Phi_2(A)$, where $\lambda:\mF^{\ast}\rightarrow \mF^{\ast}$ is a multiplicative map and $\Phi_1,\Phi_2$ are compositions of the following multiplicative maps.

\begin{enumerate}
  \item [(a).] $\Phi_R(A)=R^{-1}AR$ for some
  $n\times n$ invertible matrix $R$;
  \item [(b).] $\Phi^*(A)=(A^{\ast})^{t}$, where $A^\ast$ is the cofactor matrix of $A$ and $A^t$ is the transpose of $A$;
  \item [(c).] $\Phi_\varphi(A)=(\varphi(a_{ij}))$
for any $A=(a_{ij})\in \GL(n)$,
where $\varphi:\mF\rightarrow \mF$ is a ring homomorphism.
\end{enumerate}
\end{thm}

\begin{proof}
The proof of the theorem consists of several steps. For $j=1,\ldots,n$, let $D_{j}(k)=\diag(1,\ldots,1,k,1,\ldots,1)$, where the $j$-th
entry is $k$.

Claim (1): Up to a composition with $\Phi_R$ and a multiplication by $\det(A)$ if necessary, we have $\Phi(D_{i}(-1))=D_{i}(-1)$.

From the fact that $\Phi(D_{1}(-1))^{2}=\Phi(D_{1}(-1)^{2})=\Phi(I_{n})=I_{n}$,
composing with some $R$ when necessary, we have $\Phi(D_{1}(-1))=\left(
                                            \begin{array}{cc}
                                              -I_{m} & 0 \\
                                              0 & I_{n-m} \\
                                            \end{array}
                                          \right)
$. Since $D_{1}(-1)$ commutes with $A=\left(
                                  \begin{array}{cc}
                                    1 & 0 \\
                                    0 & A_{1}\\
                                  \end{array}
                                \right)
$, where
$A_{1}\in \GL(n-1,\mF)$, it follows that $\Phi(A)$ must
be a quasi diagonal matrix. We may therefore assume that
\[\Phi(A)=\left(
            \begin{array}{cc}
              B_{1} & 0 \\
              0 & B_{2} \\
            \end{array}
          \right)
,\textrm{where} \ B_{1}\in \GL(m,\mF), \ B_{2}\in \GL(n-m,\mF).\]
It shows that the maps $\Phi_1:\GL(n-1,\mF)\rightarrow \GL(m,\mF)$ defined by $\Phi_1(A_1)=B_1$ and $\Phi_2:\GL(n-1,\mF)\rightarrow \GL(n-m,\mF)$ defined by $\Phi_2(A_1)=B_2$ are both multiplicative maps. Then we conclude that $m=1$ or $m=n-1$ by Proposition~\ref{Ker1}, otherwise both $\Phi_1$ and $\Phi_2$ are trivials, so is $\Phi$. If $m=n-1$, we may assume that $m=1$.
Since $\Phi(D_{i}(-1))$ have the same eigenvalues as $\Phi(D_{1}(-1))$
and $\Phi(D_{i}(-1))$ commute with $\Phi(D_{j}(-1))$ for $i,j=1,\ldots,n$,
we have $\Phi(D_{i}(-1))=D_{i}(-1)$ up to a conjugation.

Claim (2): Up to a composition with $\Phi_T$, we have $\Phi(S_{ij})=S_{ij}$,
where $T$ is a diagonal matrix and $S_{ij}$ is the identity matrix with rows
$i$ and $j$ switched.

We consider the image of $S_{i,i+1}$ under $\Phi$. Note that $S_{i,i+1}$ (resp. $\Phi(S_{i,i+1})$) commutes with $D_k(-1)$ (resp. $\Phi(D_k(-1))$), for $k\neq i,i+1$, and $D_{i}(-1)S_{i,i+1}D_{i+1}(-1)=S_{i,i+1}$. Then we see that $\Phi(S_{i,i+1})=\diag(a_1,\ldots,a_{i-1},B_i,a_{i+2},\ldots,a_n)$, for some $a_i\in\mF^{\ast}$ and $B_i=\left(\begin{array}{cc}
                  0 & b_i \\
                  \frac{1}{b_i} & 0\end{array}\right)$ for some nonzero $b_i\in\mF$.
Furthermore, $S_{i,i+1}$ is similar to $D_1(-1)$, so is $\Phi(S_{i,i+1})$. Thus we have $a_i=1$. Take $T=\diag(1,b_1,b_1b_2,\ldots,b_1b_2\ldots b_{n-1})$, one have $T\Phi(S_{i,i+1})T^{-1}=S_{i,i+1}$. Replacing $\Phi$ by $T\Phi T^{-1}$, one has $\Phi(S_{ij})=S_{ij}$.

Claim (3): Up to a composition with some $\Phi_R$ and with $\Phi^*$ when necessary,
we have $\Phi(A)=(\varphi(a_{ij}))$ for any $A=(a_{ij})\in \SL(n,\mF)$, where $\varphi$
is a nontrivial ring homomorphism of $\mF$.

For invertible $A$, it is easy to see that $A$ is unipotent if and only if $A$ is conjugate to $A^n$ for any $n\in\N$. Thus $\Phi(A)$ is also unipotent.

For $j=1,\ldots,n$, let $P_{j,j+1}(k)=\diag(1,\ldots,1,\left(
                             \begin{array}{cc}
                               1 & k \\
                               0 & 1 \\
                             \end{array}
                           \right)
,1,\ldots,1)$, where the $j,j+1$ entry is $k$. Let $A=P_{12}(k)$. Then $A$ commutes with $D_i(-1)$, $i>2$. Thus one can easily get $\Phi(A)=\diag(B,1,\ldots,1)$, where $B=\left(\begin{array}{cc}a&b\\c&d\end{array}\right)$ is a $2\times 2$ unipotent matrix. By $D_1(-1)AD_1(-1)=A^{-1}$, one can get that $a=d=1$ and $bc=0$. Replacing $\Phi(A)$ by $\Phi((A^{-1})^t)=\Phi\circ(\det(A)^{-1}\Phi_2(A))$ when necessary, one can assume that $c=0$. Note that the map $A\rightarrow (A^{-1})^t$ preserves $D_i(-1)$ and $S_{ij}$. Denote $b$ by $\varphi(k)$. Then one can easily check that $$\varphi(k+l)=\varphi(k)+\varphi(l).$$ Similarly, by $S_{12}AS_{12}=A^t$, one has that $\Phi(A^t)=\Phi(A)^t$.
Since $\Phi(S_{ij})=S_{ij}$, we have that $\Phi(P_{ij}(k))=P_{ij}(\varphi(k))$.

Notice that $P_{13}(kl)=P_{12}(k)^{-1}P_{23}(l)^{-1}P_{12}(k)P_{23}(l)$, we have $$P_{13}(\varphi(kl))=P_{12}(\varphi(k))^{-1}P_{23}(\varphi(l))^{-1}P_{12}(\varphi(k))P_{23}(\varphi(l))=P_{13}(\varphi(k)\varphi(l)).$$ Hence $$\varphi(kl)=\varphi(k)\varphi(l),$$ i.e., $\varphi$ is a nontrivial ring homomorphism of $\mF$. Consequently $\varphi(1)=1$. Since $\SL(n,\mF)$ are generated by $P_{ij}(k)$, one has that $\Phi(A)=(\varphi(a_{ij}))$, for any $A=(a_{ij})\in \SL(n,\mF)$.

Now we are ready to complete the proof. For any $A\in \GL(n,\mF)$, $A=D_1(k)A_0$, where $k=\det(A)$ and $A_0\in \SL(n,\mF)$. It is easy to see that $\Phi(D_{1}(k))=s\diag(t,1,\ldots,1)$ for some $s,t\in\mF^*$. Similarly, $\Phi(D_2(k))=s\diag(1,t,1\ldots,1)$. Noticing that $D_1(k)D_2(1/k)\in \SL(n,\mF)$, then we have
$$\Phi(D_1(k)D_2(1/k))=\Phi(\diag(k,1/k,1,\ldots,1)).$$
Therefore, $\diag(t,1/t,1,\ldots)=\diag(\varphi(k),1/\varphi(k),1,\ldots,1)$, i.e., $t=\varphi(k)$. Denoting $s=\lambda(\det(A))$, where $\lambda$ is a multiplicative map on $\mF^*$, we have $\Phi(A)=\lambda(\det(A))(\varphi(a_{ij}))$, for any $A=(a_{ij})\in \GL(n,\mF)$.
\end{proof}

\section{Non-degenerate Multiplicative Maps}\label{M(n)}
In this section, we study non-degenerate multiplicative maps $\Phi$
on $M_{n}(\mF)$, i.e., there exists some degenerate matrix $A$ such that $\Phi(A)\neq 0$. The main result is the following theorem, which is well-known in some sense (see \cite{JL1968}, for example). One can easily see that it is a consequence of the proof of Theorem~\ref{GL(n)}.

\begin{thm}
Let $\Phi:M_{n}(\mF)\rightarrow M_{n}(\mF)$ be a non-degenerate multiplicative map. Then, for
 any $A=(a_{ij})\in M_{n}(\mF)$,
$$\Phi(A)=R^{-1}(\varphi(a_{ij}))R\ \ \ \mbox{   or   }\ \ \ \Phi(A)=R^{-1}(\varphi(a_{ij})^{\ast})^{t}R,$$
for some ring homomorphism $\varphi:\mF\rightarrow \mF$
and some invertible matrix $R$.
\end{thm}
\begin{proof}
If $\Phi(A)\neq 0$ for some (hence all) matrix of rank 1, letting $F_{ij}=\Phi(E_{ij})$, one has $F_{ij}F_{kl}=\delta_{jk}F_{il}$. It is a standard result that there exists a matrix $R$ such that $RF_{ij}R^{-1}=E_{ij}$. Replacing $\Phi$ by $R\Phi R^{-1}$, we may assume that $\Phi(E_{ij})=E_{ij}$.
For any $i,j$ and for any $b\in \mF$, it is easy to see that $\Phi(bE_{ij})=b'E_{ij}$, for some $b'\in\mF$. Then we have a map $\varphi:\mF\rightarrow\mF$ such that $\varphi(b)=b'$. The map $\varphi$ is independent of $i,j$ since $E_{ki}E_{ij}E_{jl}=E_{kl}$.
For any $B=(b_{ij})\in M_{n}(\mF)$, since $E_{ii}\Phi(B)E_{jj}=\Phi(E_{ii}BE_{jj})=\Phi(b_{ij}E_{ij})=\varphi(b_{ij})E_{ij}$, we have $\Phi(B)=(\varphi(b_{ij}))$ such that $\varphi(b_{ij})=b'_{ij}, (i,j=1,2,\ldots,n)$. It's an easy exercise to show that
$\varphi(1)=1$, $\varphi(kl)=\varphi(k)\varphi(l)$ and $\varphi(k+l)=\varphi(k)+\varphi(l)$
for any $k,l\in\mF$.
Thus $\varphi$ is a nontrivial ring homomorphism of $\mF$. For any $B=(b_{ij})\in M(n,\mF)$, we have $E_{ii}\Phi(B)E_{jj}=\Phi(E_{ii}BE_{jj})=\Phi(b_{ij}E_{ij})=\varphi(b_{ij})E_{ij}$. Therefore $\Phi(B)=(\varphi(b_{ij}))$.

Now assume that $\Phi(E_{ij})=0$. It is easy to see that if $\Phi(A)=0$ then $\Phi(B)=0$ provided $\rank (A)=\rank (B)$. Assume that $r$ is the least integer such that $\Phi(A)\neq 0$ if $\rank(A)=r$. Let $E_r$ be the set of diagonal idempotent matrices. Without loss of generality, we may assume that, for any $A\in E_r$,  $\Phi(A)$ is also a diagonal idempotent matrix. For any $A,B\in E_r$, $A\neq B$, then $\rank(AB)<r$ and $\Phi(AB)=\Phi(A)\Phi(B)=0$. Therefore the image $\Phi(E_r)$ has at most $n$ elements. If $r<n-1$, then $E_r$ has $C^r_n>n$ elements, and hence there exist $A,B\in E_r$ such that $\Phi(A)=\Phi(B)$, which is absurd since $\Phi(AB)=0$. Thus one can easily see that $r=n-1$ and for any $A$ with $\rank(A)=n-1$, $\rank(\Phi(A))=1$. Set $F_j=I_n-E_{jj}$. Without loss of generality, we may assume that $\Phi(F_j)=E_{jj}$. Then with the same method as in the proof of Theorem~\ref{GL(n)}, we can easily get the result.
\end{proof}

\begin{rem}
In the cases of $\mF=\Q$ or $\R$, it is easy to see that the only ring homomorphism is the identity map. For $\mF=\Q$, it is trivial. For $\mF=\R$, take any $x,y \in \R$, assume that $x-y>0$, then there exists
an element $z\in \R$ such that $z^{2}=x-y$. From
$\varphi(x-y)=\varphi(x)-\varphi(y)=\varphi(z)^{2}$,
we have $\varphi(x)>\varphi(y)$.
It follows that $\varphi(x)=x$
on $\R$.

For complex number field $\C$, there exists ring homomorphisms which are not surjective.
\end{rem}

{\bf Acknowledgements}: This work was supported by the National
Natural Science Foundation of China (Grant No. 11571182) and National Science Foundation for Young Scientists of China (Grant No. 11301287 and Grant No. 11401425).

\end{document}